\documentclass[11pt,a4paper]{amsart}

 \usepackage{todonotes}
 \usepackage{csquotes}

\newenvironment{proof:mainthm1}{\vspace{0.2cm}\paragraph{\bf\textit{Proof  of Theorem \ref{thm:main}:}}}{\hfill$\blacksquare$ }
\newenvironment{proof:corollary}{\vspace{0.2cm}\paragraph{\bf\textit{Proof  of Corollary \ref{cor:corollary}}}}{\hfill$\blacksquare$ }

\usepackage{todonotes}

\usepackage{amsfonts,amssymb,enumerate,bm,hhline,makecell,array,mathtools}
\usepackage{mathrsfs}
\usepackage{comment}
\usepackage[normalem]{ulem}
\usepackage[colorlinks=true,citecolor=blue!70!black, urlcolor=blue!70!black, pagebackref, linkcolor=blue!70!black]{hyperref}
\usepackage{tikz-cd}
\usepackage{amscd}
\usepackage[shortlabels]{enumitem}
\setlist[itemize]{noitemsep,nolistsep}
\usepackage{tensor}
\usepackage{tikz}
\usepackage{setspace}
\usetikzlibrary{matrix,arrows}
\usepackage{extarrows}
\usepackage[toc,page]{appendix}
\usepackage[all,ps,cmtip,rotate]{xy} 
\usepackage{color}
\usepackage{hyperref}
\hypersetup{
    colorlinks,
    citecolor = green!80!black,
    linkcolor=red,
    urlcolor=magenta
}
\usepackage{lmodern}
\usepackage[all]{xy}
\usepackage{comment}

\usepackage{enumitem}
\setlist[itemize]{noitemsep,nolistsep}
\setlist[enumerate]{noitemsep,nolistsep}
\usepackage{colortbl} 

\usepackage{geometry}
\allowdisplaybreaks

\def\Z{{\bf Z}}

\def\C{{\bf C}}

\def\R{{\bf R}}
\def\Q{{\bf Q}}
\def\P{{\bf P}}

\def\hk{hyper-K\"ahler}

\def\phi{\varphi}

\def\cA{\mathcal{A}}

\def\cD{\mathcal{D}}

\def\cG{\mathcal{G}}

\def\cP{\mathcal{P}}

\def\cS{\mathcal{S}}

\def\cX{\mathscr{X}}
\def\cY{\mathcal{Y}}
\def\cZ{\mathcal{Z}}

\def\llra{\hbox to 10mm{\rightarrowfill}}
\def\lllra{\hbox to 15mm{\rightarrowfill}}

\def\llla{\hbox to 10mm{\leftarrowfill}}
\def\lllla{\hbox to 15mm{\leftarrowfill}}

\DeclareMathOperator{\Bir}{Bir}

\DeclareMathOperator{\Exc}{Exc}

\DeclareMathOperator{\Hdg}{Hdg}

\DeclareMathOperator{\Mon}{Mon}
\DeclareMathOperator{\Mov}{Mov}

\DeclareMathOperator{\Pic}{Pic}
\DeclareMathOperator{\Pos}{Pos}

\DeclareMathOperator{\rank}{rank}

\DeclareMathOperator{\tO}{\mathrm{O}}

\def\llra{\hbox to 10mm{\rightarrowfill}}
\def\lllra{\hbox to 15mm{\rightarrowfill}}

\def\subset{\subseteq}

\newtheorem{lem}{Lemma}[section]

\newtheorem{thm}[lem]{Theorem}
\newtheorem{cor}[lem]{Corollary}
\newtheorem{prop}[lem]{Proposition}

\theoremstyle{definition}
\newtheorem{defn}[lem]{Definition}
\newtheorem{rem}[lem]{Remark}

\newtheorem{ex}[lem]{Example}

\theoremstyle{remark}
\newtheorem*{remark*}{Remark}
\newtheorem*{note*}{Note}

\def\Lkkk[#1]{{\Lambda_{\KKK^{[#1]}}}}
\def\kkk[#1]{{\KKK^{[#1]}}}
\DeclareMathOperator{\KKK}{{K3}}

\def\sss[#1]{{S^{[#1]}}}

\def\setminus{\smallsetminus}

\definecolor{orange}{rgb}{1,0.55,0}

\definecolor{pink}{rgb}{1,0.65,0.79}

\begin{document}

\title[Positive algebraic entropy]{Algebraic entropy of birational automorphism groups in families of hyper-Kähler manifolds} 

\address{}
\email{}
    \author[F.\ A.\ Denisi]{Francesco Antonio Denisi}
    \address{F.\ A.\ Denisi, Fachrichtung Mathematik, Campus, Gebäude E2 4, Universität des Saarlandes, 66123 Saarbrücken, Deutschland}
    \email{denisi@math.uni-sb.de}

    \author[K.\ Oguiso]{Keiji Oguiso}
     \address{K.\ Oguiso, Mathematical Sciences, the University of Tokyo, Meguro Komaba 3-8-1, 
     Tokyo, Japan, and National Center for Theoretical Sciences, 
     Mathematics Division, National Taiwan University, 
     Taipei, Taiwan}
    \email{oguiso@ms.u-tokyo.ac.jp}
    
    \author[C.\ Onorati]{Claudio Onorati}
    \address{C.\ Onorati, Università della Calabria, Department of Mathematics and Computer Science, via Pietro Bucci, 87036 Arcavacata di Rende (CS), Italy} 
    \email{claudio.onorati@unical.it}

    \author[F.\ Rizzo]{Francesca Rizzo}
    \address{F.\ Rizzo, IRMA, Université de Strasbourg 7, rue René Descartes, 67000 Strasbourg}
    \email{francesca.rizzo@unistra.fr}

    \author[S.\ Viktorova]{Sasha Viktorova}
    \address{S.\ Viktorova, National Center for Theoretical Sciences, Mathematics Division, National Taiwan University, Taipei, Taiwan}
    \email{sasha.viktorova@ncts.ntu.edu.tw}

\subjclass[2020]{Primary 14J42; Secondary 14E07, 14J27, 37F80}

\begin{abstract}
In this short note, we study the behavior of the algebraic entropy of birational automorphism groups in families of projective hyper-Kähler manifolds.
\end{abstract} 

\maketitle

\section{Introduction}\label{sect0}

In the following note, we continue the study of the behavior of birational automorphism groups in a family of hyper-Kähler manifolds started in \cite{DORV26}, from the viewpoint of algebraic dynamics..

Let $X$ be a hyper-Kähler manifold and $g$ a birational automorphism of $X$. We define \textit{algebraic entropy} at $H^2$-level as
\[
    h(g):=\log\lambda_1(g),
\]
where $$\lambda_1(g)=\limsup_{n\to\infty} \Big\| (g^n)^*|_{H^{1,1}(X,\C)}\Big\|^{\frac{1}{n}}$$ is the first dynamical degree of $g$. We refer to \cite{DinhSibony05} for generalities of dynamical degrees of rational maps. Notice that a birational automorphism of a hyper-K\"ahler manifold is always an isomorphism in codimension 1, so that its action on $H^2(X,\Z)$ behaves like that of a regular automorphism. The first dynamical degree coincides then with the maximum of the absolute values of the eigenvalues of the action of $g$ on $H^{1,1}(X,\C)$. We refer to \cite{Oguiso07,Oguiso09,LoBianco} for some general properties of algebraic entropy of hyper-K\"ahler manifolds. 
If $g$ has finite order, then clearly $h(g)=0$; on the other hand, there are examples of birational automorphisms $g$ of infinite order with $h(g)=0$ (e.g.\ \cite{Oguiso06}). Therefore, having positive entropy is a stronger property than being of infinite order. 

Our main result is the following generalization of \cite[Theorem~1.2]{DORV26} and \cite[Corollary~1.2]{ASV26}.

\begin{thm}\label{thm:main}
    Let $\pi \colon \cX\to \Delta$ be a non-isotrivial polarized family of hyper-Kähler manifolds with $b_2 \ge 6$ over a small disk $\Delta$.
    Then there exists a dense subset $\mathcal{D}\subset \Delta$, in the Euclidean topology, such that for $t\in \mathcal{D}$, the fiber $\cX_t$ carries a birational automorphism of positive algebraic entropy.
\end{thm}

\begin{rem}\label{rmk:counter-example for K3 surfaces}
    The projectivity assumption cannot be relaxed. In fact, by \cite{Oguiso08} it is possible to construct a family of K3 surfaces over a small disk where the very general fiber is non-projective, and every fiber has a finite automorphism group. More precisely, by \cite[Theorem~1.5]{Oguiso08} any K3 surface of Picard rank $1$ has finite automorphism group, and by \cite[Corollary~1.6]{Oguiso08} there are at most countably many K3 surfaces with Picard rank $0$ and non-trivial automorphism group. We obtain the desired family by taking a small disk in the Kuranishi space of K3 surfaces that avoids K3 surfaces of Picard rank greater than $1$ and the countably many K3 surfaces of Picard rank $0$.
\end{rem}

As a corollary, we get the same density result for any family of projective hyper-K\"ahler manifolds.

\begin{cor}\label{cor:corollary}
    Let $\pi \colon \cX\to B$ be a non-isotrivial family of projective hyper-Kähler manifolds with $b_2\ge 6$ over an irreducible complex analytic space $B$. Then there exists a dense subset $\mathcal{D}\subset B$, in the Euclidean topology, such that for $t\in \mathcal{D}$, the fiber $\cX_t$ carries a birational automorphism of positive algebraic entropy.
\end{cor}

\begin{rem}
    If we replace hyper-K\"ahler manifolds with strict Calabi--Yau manifolds, the result does not hold in general. For example, any smooth hypersurface of degree $n+1$ in $\P^n$ has finitely many birational self-maps when $n\geq4$. This shows the need to work with hyper-K\"ahler manifolds.
\end{rem}

\begin{rem}
    Combining our Corollary~\ref{cor:corollary} with \cite[Main Theorem.(3)]{LoBianco}, we infer that in a smooth family of projective hyper-K\"ahler manifolds with $b_2\ge 6$ there is a dense subset of the base where the corresponding members of the family have a birational automorphism with a Zariski-dense generic orbit.
\end{rem}

\begin{ex}
    In \cite[Example 1.4.(1)]{DORV26} the authors exhibit a family of hyper-K\"ahler fourfolds whose members have an infinite birational automorphism group. By Theorem~\ref{thm:main} and Remark~\ref{rmk:rho 2}, we can upgrade this example by noting that, in fact, all members of the family have a birational automorphism of positive algebraic entropy.
\end{ex}

\subsection*{Acknowledgments}
Much of this project was carried out when F.\ A.\ Denisi, C.\ Onorati, F.\ Rizzo, and S.\ Viktorova visited K.\ Oguiso at the University of Tokyo. We thank the University of Tokyo and the grants JSPS 25H00587 and 25K21992 (P.I.\ K.\ Oguiso) for making this collaboration possible. Denisi is supported by the Deutsche Forschungsgemeinschaft (DFG, German Research Foundation) – Project ID 286237555 (TRR 195). Onorati acknowledges the support of \enquote{National Group for Algebraic and Geometric Structures, and their Applications} (INdAM--GNSAGA). Rizzo acknowledges funding from the European Research Council under the European Union's Horizon 2020 research and innovation program (ERC-2020-SyG-854361-HyperK). Oguiso and Viktorova thank NCTS for its support. Onorati and Rizzo are part of the INdAM--GNSAGA.

\section{Preliminaries}

\subsection{Parabolic reflective lattices}\label{section:lattices}
A lattice is a finitely generated free abelian group $L$, endowed with an integer-valued, nondegenerate symmetric bilinear form. In what follows, we will mostly deal with \emph{hyperbolic} lattices, that is, lattices for which, if the rank is $n$, the associated bilinear form has signature $(1,n-1)$. As usual, we say that a vector $v\in L$ is \emph{isotropic} if $v^2=0$.

Let $L$ be a hyperbolic lattice. We denote by $\tO(L)$ the group of isometries of $L$, that is, the group of automorphisms of $L$ preserving the bilinear form.
The positive cone $\Pos(L)$ of $L$ is one of the two connected components of the set
\[
\{x\in L_\R:=L\otimes \R\mid x^2>0\}.
\]
We have an index $2$ subgroup of $\tO(L)$, defined as
\[
\tO^+(L):=\{g\in \tO(L)\mid g(\Pos(L))=\Pos(L)\}.
\]
Denote by $W(L)$ the subgroup of $\tO^{+}(L)$ generated by the reflections $s_{\alpha}$, where
\[
s_{\alpha}(x)=x-\frac{2x\cdot\alpha}{\alpha^2}\alpha
\]
for every $x\in L$, whenever this formula defines an element of $\tO^{+}(L)$.
Let $W\subseteq W(L)$ be a reflection group and $\mathcal{M}\subseteq L_{\R}$ a fundamental domain of $W$. Define
\[
    \cA(\mathcal{M}) := \{g\in\tO^{+}(L) \mid g(\mathcal{M}) = \mathcal{M} \}.
\]
The following definition, adapted from \cite{Nikulin96}, will play an important role in this paper.

\begin{defn}[\protect{\cite[Definition 1.1.2]{Nikulin96}}]\label{defn:parabolic}
    Let $L$ be a hyperbolic lattice. A reflection group $W\subset W(L)$ is called \textit{parabolic} if the group $\cA(\mathcal{M})$ is infinite, $\big[\tO^{+}(L)\colon W \rtimes \cA(\mathcal{M})\big]<+\infty$ and there is an element $0\neq v\in L$ such that $v^2=0$ and $g(v)=v$, for any $g\in\cA(\mathcal{M})$.
    The lattice $L$ is called \emph{parabolic reflective} if $\tO^{+}(L)$ contains a parabolic reflection group $W$.
\end{defn}

\begin{thm}[\protect{\cite[Theorem 1.1.3]{Nikulin96}}]\label{thm:Nikulin96}
    For each fixed $\rank(L)\geq 3$, the set of parabolic reflective hyperbolic lattices $L$ is finite up to isomorphism and similarity $L\mapsto L(m)$, where $L(m)$ is the lattice obtained by multiplying the bilinear form of $L$ by $m$.
\end{thm}

Let $g\in \tO^+(L)$. The dynamical degree $\lambda_1(g)$ of $g$ is the maximum absolute value of the eigenvalues of the $\C$-linear extension $g_{\C}$ of $g$ to $L_\C:=L\otimes \C$. By abuse of notation, we call
\[
h(g):=\log\lambda_1(g)
\]
the \textit{algebraic entropy} of $g$. Since $g\in \tO^+(L)$, we have $\det (g) = \pm1$, hence, $h(g)\geq 0$.

\begin{thm}[\protect{\cite[Theorem 2.1]{Oguiso07}}]\label{thm:Oguiso07}
    Let $L$ be a hyperbolic lattice of rank $r$ and let $G$ be a subgroup of $\tO^{+}(L)$. Assume that $G$ is an infinite group of zero algebraic entropy, i.e., $|G|=+\infty$ and $h(g)=0$ for all $g\in G$. Then: 
    \begin{enumerate}[label = \rm{(\arabic*)}]
        \item\label{oguiso-part1} there is an isotropic vector $0\neq v\in L$ such that $g(v)=v$ for all $g\in G$; 
        \item\label{oguiso-part2} the group $G$ is virtually abelian. 
    \end{enumerate}
\end{thm}

\subsection{Hyper-Kähler manifolds of zero algebraic entropy}
Let $X$ be a hyper-K\"ahler manifold. Recall that the torsion free abelian group $H^2(X,\Z)$ is endowed with a non-degenerate quadratic form $q_X$ turning it into a lattice of signature $(3,b_2(X)-3)$. 
Throughout this sub-section we assume that $b_2(X)\geq5$.
If $X$ is projective, the Picard lattice $(\Pic(X),q_X)$ is hyperbolic of signature $(1,\rho(X)-1)$, where $\rho(X)$ is the Picard rank. We refer to \cite[Corollary~23.11]{GHJ03} for a general treatment of these facts. From now on, all manifolds are assumed to be projective.

Recall that a prime divisor $E$ on a projective hyper-Kähler manifold $X$ is exceptional if $q_X(E)<0$. By \cite{Markman:Galois}, the reflection $s_E$ is integral (and in fact a monodromy operator). We denote then by $W_{\Exc}(X)$ the subgroup of $\tO^+(H^2(X,\Z))$ generated by reflections with respect to classes of stably prime exceptional divisors. These isometries preserve the Hodge structure, so they also act on $\Pic(X)$. As usual, we abuse notation and whenever needed we consider the restriction of the action of a Hodge isometry to $\Pic(X)$ without specifying it. The closure of the movable cone $\overline{\Mov}(X)\subseteq\Pic(X)\otimes\R$ is a fundamental domain for the action of $W_{\Exc}(X)$ (\cite[Lemma~6.22]{Markman11}). Finally, we denote by $\Mon^2_{\Bir}(X)$ the image of the action of $\Bir(X)$ on $H^2(X,\Z)$; recall that this action has finite kernel (\cite[Proposition~9.1]{Huybrechts:Basic}). By \cite[Theorem~6.18]{Markman11}, we have a decomposition
\begin{equation}\label{eqn:W x Bir}
   W_{\Exc}(X) \rtimes \text{Mon}^2_{\Bir}(X) = \text{Mon}^2_{\Hdg}(X)
\end{equation}
and the group $\Mon^2_{\Bir}(X)$ is identified with the stabiliser of the movable chamber, i.e.\ $\Mon^2_{\Bir}(X)=\cA\big(\overline{\Mov}(X)\big)\subseteq\tO^+(\Pic(X))$ in the notations of Section~\ref{section:lattices}.

In the following we say that $\Bir(X)$ has zero algebraic entropy if $h(g)=0$ for every $g\in\Bir(X)$.

\begin{lem}\label{lem:NullEntropyParabolic}
    Let $X$ be a projective hyper-Kähler manifold such that $\Bir(X)$ is infinite of zero algebraic entropy. Then the lattice $\Pic(X)$ is parabolic reflective.
\end{lem}
\begin{proof}
   We show that the group $W_{\Exc}(X)$ is parabolic reflective. Since $\Bir(X)$ is infinite, the group $\Mon^2_{\Bir}(X)$ is also infinite. Moreover, the group 
   $\Mon^2_{\Hdg}(X)$
   has finite index in $\tO^+(\Pic(X))$ (\cite[Lemma~6.23]{Markman11}). Together with (\ref{eqn:W x Bir}), this already proves the first two conditions of Definition~\ref{defn:parabolic}. Finally, Theorem~\ref{thm:Oguiso07}.\ref{oguiso-part1} concludes the proof.
\end{proof}

The following proposition generalizes \cite[Proposition~2.6]{DORV26}.

\begin{prop} \label{prop:finlattices}
    Let $X$ be a projective hyper-Kähler manifold with $\rho(X)\ge 3$ such that $\Bir(X)$ has zero algebraic entropy. Then there are finitely many lattices representing $\Pic(X)$.
\end{prop}
\begin{proof}
    We only need to consider the case where $\Bir(X)$ is infinite and of zero algebraic entropy, as \cite[Proposition~2.6]{DORV26} deals with $\Bir(X)$ finite. By Lemma~\ref{lem:NullEntropyParabolic} and Theorem~\ref{thm:Nikulin96}, up to rescaling, there are only finitely many isometry classes of lattices that can occur as $\Pic(X)$. To conclude, as in the proof of \cite[Proposition~2.6]{DORV26}, we only need to show that $X$ carries at least one prime exceptional divisor or, equivalently, that $W_{\Exc}(X)$ is not trivial.
    By Theorem~\ref{thm:Oguiso07}.\ref{oguiso-part2} the group $\Mon^2_{\Bir}(X)$ is virtually abelian. Suppose by contradiction that $W_{\Exc}(X)$ is trivial. Then, by (\ref{eqn:W x Bir}), also the group $\text{Mon}^2_{\Hdg}(X)$, and thus $\tO^+(\Pic(X))$, is virtually abelian. However, the group $\tO^+(\Pic(X))$ is not virtually abelian when $\rho(X)\geq 3$, by the Borel--Harish-Chandra theorem (\cite[Theorem~7.8]{BHC}) and the Borel density theorem (\cite[Section~4.3]{Borel}), giving a contradiction.
\end{proof}

\begin{rem}
    Explicitly describing the finitely many lattices in the proposition above is an interesting lattice-theoretic problem. For example, Yu recently achieved this for K3 surfaces with an infinite automorphism group of zero algebraic entropy in \cite{XunYu}.
\end{rem}

If $h\in\Pic(X)$ is the class of a polarization on $X$, we denote by $\Mon^2_h(X)$ the subgroup of $\Mon^2(X)$ preserving the class $h$. By \cite[Remark~1.1]{Huyb18} and Proposition \ref{prop:finlattices} above, up to the action of $\Mon^2_h(X)$, there are only finitely many sublattices of $h^\perp\subseteq H^2(X,\Z)$ that can occur as the Picard lattices of $h$-polarized hyper-Kähler manifolds $X$ with $\rho(X)\ge 3$ for which $\Bir(X)$ has zero algebraic entropy. We denote by $L_1,\dots,L_n$ a set of representatives of these isometry classes.

The following definition is analogous to \cite[Definition~2.2]{ASV26}.
\begin{defn}\label{defn:Z}
Consider the period space $\cP_{h^\perp}:=\cD_{h^\perp}/\Mon^2_h(X)$ of $h$-polarized hyper-Kähler manifolds, with
    \[
    \cD_{h^\perp}:=\{x \in \mathbf{P}(h^{\perp} \otimes_{\mathbf{Z}} \mathbf{C}) \mid q_X(x)=0,\quad \text{and}\quad q_X(x+\overline{x})>0\}.
    \]
    We denote by $\cZ_i$ the image of $\P(L_i^\perp\otimes \C)\cap \cD_{h^\perp}$ in the period domain $\cP_{h^\perp}$.
\end{defn}

The same proof of \cite[Theorem~1.1(1)]{ASV26} gives the corresponding statement.
\begin{prop}\label{prop:DivisorsZi}
    For any $h$-polarized hyper-Kähler manifold $X$ with $\rho(X)\ge 3$ such that $\Bir(X)$ has zero algebraic entropy, there exists a unique index $i$ such that the period point $\wp([X, h])$ is a very general point of the subvariety $\cZ_i\subset \cP_{h^\perp}$.
\end{prop}
In the proposition above, $\wp$ denotes the period map from the moduli space of $h$-polarized hyper-K\"ahler manifolds deformation equivalent to $X$ to the period domain.

\subsection{The divisors $\cY_j$}

This section covers \cite[Section~2.3]{ASV26}, and we give details for completeness.

\begin{prop}[{\cite[Proposition~2.6]{ASV26}}]\label{prop:ReticoliMBM}
    Let $(\Lambda, q)$ be the BBF lattice of a hyper-Kähler manifold $X$ with $b_2(X)\ge 6$ and let $h\in \Lambda$ be a primitive vector of positive square. 
    There exists a family of primitive vectors $\{v_i\}$ in $h^\perp$ such that 
    \begin{itemize}
        \item $\lim_{i\to \infty} q(v_i) = -\infty$; 
        \item the saturation $M_i$ of the sublattice $\langle h, v_i\rangle \subset \Lambda$ is a primitive rank-2 anisotropic lattice that contains no MBM classes.
    \end{itemize}
\end{prop}

\begin{proof}
    We denote by $h^\perp$ the sublattice of $\Lambda$ orthogonal to $h$.
    By Meyer's theorem (see for example \cite[Corollary~2 at p.\ 43]{Serre:Arithmetique}), since the lattice $h^\perp$ is indefinite and has rank at least 5, there exists an isotropic vector $e\in h^\perp$. Let $f\in h^\perp$ be a primitive vector such that the product $m\coloneqq q(e, f)$ is a positive generator of the ideal $q(e, h^\perp)\subset \Z$. Consider the integral lattice $L\subset h^\perp\otimes \Q$ spanned by $h^\perp$ and $\frac{e}m$, and the sublattice $S$ generated by $\frac{e}m$ and $f$. The lattice $S$ is unimodular, hence primitive in $L$. The vector
    $$
        u_a\coloneqq \left(1+2daf^2\right)\frac{e}m -4da f \in S
    $$
    is primitive of square $-8da$, and the lattice $\Z h\oplus \Z u_a$ is given by the quadratic form $\begin{pmatrix}
        d & 0 \\
        0 & -8ad
    \end{pmatrix}$. This lattice is denoted by $L_{8a}(d)$ in \cite{ASV26}. 

    Let $N$ be the bound for the square of MBM classes for the hyper-Kähler manifolds which are deformation equivalent to $X$, and let $T$ be the index of the sublattice $\Z h\oplus h^\perp$ in $\Lambda$. \cite[Proposition 2.3]{ASV26} applied with $n=NT^2$ and $b=8$, gives an increasing sequence of positive integers $\{a_i\}$ such that the lattice $\Z h\oplus \Z u_{a_i}$ does not represent $0, -1, -2, \dots, -NT^2$. 

    Let $v_i$ be a primitive generator of $\Z u_{a_i}\cap h^\perp$. Since the index of $h^\perp$ in $L$ is $m$, $q(v_i)\le  q(u_{a_i})$, hence the first condition is satisfied. Therefore, we only need to verify that the saturation $M_i$ of the sublattice $\Z h\oplus \Z v_{i}\subset \Lambda$ is anisotropic and contains no MBM class. Note that since $h$ and $v_i$ are primitive vectors, the index of $\Z h\oplus \Z v_{i}$ in $M_i$ is at most $T$. Hence, the lattice $M_i$ does not contain elements of square $0, -1, \dots, -N$, which concludes the proof.
\end{proof}

\begin{defn}[{\cite[Definition~2.7]{ASV26}}]\label{def:Divisors}
    We denote by $\cY_i$ the image of $\P(M_i^\perp\otimes \C)\cap \cD_{h^\perp}$ in the period domain $\cP_{h^\perp}$.
\end{defn}

By \cite[Section~2.4]{ASV26}, the collection $\{\mathcal{Y}_i\}_i$ is strongly dense, i.e.\ for any non-constant holomorphic morphism $f\colon\Delta\to\cP_{h^\perp}$, the subset $f^{-1}(\bigcup_i \cY_i)$ is dense in $\Delta$. Here $\Delta$ is the complex unit disk.
\begin{rem}
   A general point of $\cY_j$ corresponds to a hyper-K\"ahler manifold with an automorphism of infinite order (cf.\ \cite[proof of Theorem~1.1.3]{ASV26}). Since we are interested in birational self-maps, it would have been enough to consider MBM classes corresponding to stably prime exceptional divisors only; nevertheless, this would not simplify the proof, so we keep the construction as in \cite{ASV26}.
\end{rem}

\section{Proof of Theorem~\ref{thm:main}}
Let $\pi\colon\mathscr{X}\to \Delta$ be a non-isotrivial family of projective \hk\ manifolds over a small disk, with marking $\tau \colon R^2\varphi_*\Z_\mathscr{X}\to \Lambda\times \Delta$.

Set $\Lambda_t \coloneqq \tau_t(N^1(\mathscr{X}_t))$. By \cite[Section~2]{Oguiso03}, there exists a primitive lattice $\Lambda_0\subset \Lambda$ and an uncountable dense subset $\cG\subset \Delta$ such that:
\begin{enumerate}
    \item $\cS\coloneqq \Delta\setminus \cG$ is countable and dense;
    \item $\Lambda_0\subset \Lambda_t$ for all $t\in \Delta$;
    \item $\Lambda_0 = \Lambda_t$ for all $t\in \cG$;
    \item $\Lambda_0\subsetneq \Lambda_t$ for all $t\in \cS$.
\end{enumerate}
We will say that the points in $\cG$ are \emph{very general}, while those in $\cS$ are \emph{special}. Note that the Picard number $\rho(\mathscr{X}_t)$ determines whether the point $t$ is very general or special.

The proof of Theorem~\ref{thm:main} is the same as the proof of \cite[Corollary~1.2]{ASV26}. We re-propose it here for completeness. Recall from Definitions~\ref{defn:Z} and \ref{def:Divisors} the subvarieties $\mathcal{Z}_j$ and the divisors $\cY_i$, respectively.

\begin{proof:mainthm1}
Assume that one very general fiber of $\pi\colon \cX \to \Delta$ carries a birational automorphism of positive algebraic entropy. By \cite[Theorem~1.1]{DORV26}, up to shrinking $\Delta$, there exists a subgroup $G^{0}<\Mon^2_{\Bir}(\cX_t)$ of finite index for every $t$ that is very general. It follows that $G^0$ carries a birational automorphism of positive algebraic entropy and thus $\Mon^2_{\Bir}(\cX_t)$ carries a birational automorphism of positive algebraic entropy for every very general $t$. 

Hence we may assume that $\Bir(\cX_t)$ has zero algebraic entropy for all very general $t$. Consider the classifying morphism $\psi \colon \Delta\rightarrow \cP_{h^\perp}$. Let $L$ be the lattice that represents the isomorphism class of $\Pic(\cX_t)$ for a very general $t\in \Delta$. We divide the proof into three cases.

First, assume that $\rank(L)\ge 3$. Since $\Bir(\cX_t)$ has zero algebraic entropy for $t$ very general, by Proposition~\ref{prop:DivisorsZi} there exists a unique index $i$ such that $\psi(t)$ is a very general point of $\cZ_i$ for all very general $t$. Therefore, $\psi(\Delta)$ is contained in $\cZ_i$, and is not contained in any $\cZ_j$ of dimension smaller than $\dim(\cZ_j)$. 

We set 
$$
J \coloneqq \{j \le n\mid \dim(\cZ_j)<\dim(\cZ_i)\}.
$$ 
Consider the set $\cS$ of points $t$ where $\rho(\cX_t)>\rank(L)$. Note that $\psi(\cS)$ is not contained in $\bigcup_{j\in J} \cZ_j$ since $\cS$ is dense in $\Delta$ by \cite{Oguiso03} and the subvarieties $\cZ_j$ are irreducible. Since $J$ is finite, the set
$$
\Delta' = \Big\{t\in \cS \mid \psi(t)\not\in \bigcup_{j\in J} \cZ_j\Big\}
$$
is dense in $\Delta$. Moreover, for every $t\in \Delta'$, the fiber $\cX_t$ contains a birational automorphism of positive algebraic entropy, since otherwise Proposition~\ref{prop:DivisorsZi} would yield a contradiction. This concludes the proof in the case where $\rank(L)\ge 3$.

Consider now the case $\rank(L)= 2$, and consider again the dense subset $\cS$ of points $t$ such that $\rho(\cX_t) > 2$. As above, $\psi(\cS)$ cannot be contained in $\cup_{i} \cZ_i$, and we conclude by defining $\Delta' := \{t\in \cS\mid \psi(t)\notin \cup_{j} \cZ_j\}$.

Finally, consider the case $\rank(L)= 1$, and the divisors $\cY_j$ constructed in Definition~\ref{def:Divisors}. The strong density of the collection $\{\cY_i\}$ (see \cite[Theorem~1.1]{ASV26}) implies that the preimage $\Delta'' = \psi^{-1}(\cup_j \cY_j)$ is dense in $\Delta$. Note that by construction of $\cY_i$, for every $t\in \Delta''$ such that $\rho(\cX_t)$ equals two, $\Bir(\cX_t)$ has positive algebraic entropy. In fact, since it contains no MBM classes by construction, it is infinite, and by
Theorem~\ref{thm:Oguiso07}, since it does not contain isotropic classes, it must have positive entropy. Then we conclude as above by defining $\Delta' := \{t\in \Delta''\mid \psi(t)\notin \bigcup_{i} \cZ_i\}$.
\end{proof:mainthm1}

\begin{proof:corollary}
Suppose that the locus $\mathcal{D}$ of points $s\in B$ such that $\cX_s$ has birational automorphisms of positive algebraic entropy is not dense. After taking the normalization $B^{\nu}\to B$ and then a resolution $\widetilde{B}\to B^{\nu}$, we may assume that the base $B$ of the family is smooth. Since $\mathcal{D}$ is not dense, we may pick its closure $\overline{\mathcal{D}}\subsetneq B$ and work with the restriction $\pi_{B'} \colon \cX':=\cX\times_B B'\to B'$, where $B':= B \setminus \overline{\mathcal{D}}$. Note that the restriction of $\pi_{B'}$ to any connected component of $B'$ stays non-isotrivial. Indeed, let $B''$ be a connected component of $B'$; since $B''$ is open in $B'$, it is also open in $B$. There is a morphism of locally free sheaves $f\colon T_{B}\to R^1\pi_\ast T_{\cX/B}$ induced by $\pi$, and if $\pi_{B''}$ were isotrivial, then the image of $f$ would be a torsion subsheaf of $R^1\pi_\ast T_{\cX/B}$. Since the latter is locally free, the morphism $f$ must be trivial, which contradicts the non-isotriviality of $\pi$.

By working on a connected component of $B'$, we can now suppose that $B'$ is connected and $\pi_{B'}\colon \cX'\to B'$ is non-isotrivial. Then there exists a point $0\in B'$ at which the Kodaira--Spencer map $T_{B',0}\to H^1(\cX_0,T_{\cX_0})$ is nonzero. Since $B'$ is a complex manifold, we can find an arbitrarily small disk $\Delta_{\epsilon}$ passing through $0$ such that the family $\pi_{\Delta_{\epsilon}}\colon \cX'\times_{B'}\Delta_{\epsilon}\to \Delta_{\epsilon}$ is projective and non-isotrivial. By construction, the birational automorphisms of the fibers all have algebraic entropy zero, and this contradicts Theorem \ref{thm:main}.
\end{proof:corollary}

\begin{rem}\label{rmk:rho 2}
    When $\rho(X)=2$, the group $\Bir(X)$ is finite if and only if it has zero algebraic entropy. Indeed, if we assume that $\Bir(X)$ is infinite and of zero algebraic entropy, then the lattice $\Pic(X)$ contains an isotropic vector by Theorem~\ref{thm:Oguiso07}.\ref{oguiso-part1}. The result then follows from \cite[Theorem~1.3]{Oguiso14}.
\end{rem}


\bibliography{biblio}

@article {Oguiso14,
    AUTHOR = {Oguiso, Keiji},
     TITLE = {Automorphism groups of {C}alabi-{Y}au manifolds of {P}icard
              number 2},
   JOURNAL = {J. Algebraic Geom.},
  FJOURNAL = {Journal of Algebraic Geometry},
    VOLUME = {23},
      YEAR = {2014},
    NUMBER = {4},
     PAGES = {775--795}
}

@article {Huyb18,
    AUTHOR = {Huybrechts, Daniel},
     TITLE = {Finiteness of polarized {K}3 surfaces and hyperk\"ahler
              manifolds},
   JOURNAL = {Ann. H. Lebesgue},
  FJOURNAL = {Annales Henri Lebesgue},
    VOLUME = {1},
      YEAR = {2018},
     PAGES = {227--246}
}

@article {Oguiso03,
    AUTHOR = {Oguiso, Keiji},
     TITLE = {Local families of {$K3$} surfaces and applications},
   JOURNAL = {J. Algebraic Geom.},
  FJOURNAL = {Journal of Algebraic Geometry},
    VOLUME = {12},
      YEAR = {2003},
    NUMBER = {3},
     PAGES = {405--433}
}

@book{GHJ03,
    author  = "Mark Gross and Daniel Huybrechts and Dominic Joyce",
    title   = "Calabi-Yau Manifolds and Related Geometries. Lectures at a Summer School in Nordfjordeid, Norway, June 2001",
    year    = "2003",
    publisher  = "Springer-Verlag"
}

@article{Nikulin96,
    author = {Viacheslav Valentinovich Nikulin},
    title = {Reflection groups in {L}obachevskii spaces and the denominator identity for {L}orentzian {K}ac--{M}oody algebras},
    journal = {Izv. Math.},
    year = {1996},
    volume = {60},
    issue = {2},
    pages = {305--334}
}

@article {Borel,
    AUTHOR = {Borel, Armand},
     TITLE = {Density properties for certain subgroups of semi-simple groups
              without compact components},
   JOURNAL = {Ann. of Math. (2)},
  FJOURNAL = {Annals of Mathematics. Second Series},
    VOLUME = {72},
      YEAR = {1960},
     PAGES = {179--188}
}

@article{BHC,
    AUTHOR = {Borel, Armand and Harish-Chandra},
     TITLE = {Arithmetic subgroups of algebraic groups},
   JOURNAL = {Ann. of Math. (2)},
  FJOURNAL = {Annals of Mathematics. Second Series},
    VOLUME = {75},
      YEAR = {1962},
     PAGES = {485--535}
}

@article{DinhSibony05,
    author ={Dinh, Tien-Cuong and Sibony, Nessim},
    title = {Une borne supérieure pour l’entropie
topologique d’une application rationnelle},
    journal = {Ann. of Math. (2)},
    year = {2005},
    volume = {161},
    pages = { 1637--1644}
}

@article{LoBianco,
    author ={Lo Bianco, Federico},
    title ={On the Primitivity of birational Transformations of Irreducible Holomorphic Symplectic Manifolds},
    journal ={Inter. Math. Res. Not. IMRN},
    volume={2019},
    number={1},
    pages={1--32},
    year ={2017}
}

@article{Markman:Galois,
    author = {Markman, Eyal},
    title = {Modular {G}alois covers associated to symplectic resolutions of singularities},
    journal = {J. Reine Angew. Math},
    year = {2010},
    volume = {644},
    pages = {189--220}
}

@incollection {Markman11,
    AUTHOR = {Markman, Eyal},
     TITLE = {A survey of {T}orelli and monodromy results for
              holomorphic-symplectic varieties},
 BOOKTITLE = {Complex and differential geometry},
    SERIES = {Springer Proc. Math.},
    VOLUME = {8},
     PAGES = {257--322},
 PUBLISHER = {Springer, Heidelberg},
      YEAR = {2011}
}

@article{Oguiso06,
    author = {Oguiso, Keiji},
    title = {Tits alternative in hyperk{\"a}hler manifolds},
    journal = {Math. Res. Lett.},
    year = {2006},
    volume = {13},
    number = {2},
    pages ={307--316}
}

@article{Oguiso07,
    author = {Oguiso, Keiji},
    title = {Automorphisms of hyperk\"{a}hler manifolds in the view of topological entropy},
    journal = {Contemporary Mathematics},
    year = {2007},
    volume = {422},
    pages = {173--185}
}

@article{Oguiso08,
    author ={Oguiso, Keiji},
    title ={Bimeromorphic automorphism groups of non-projective hyperkähler manifolds -- a note inspired by {C.T. M}c{M}ullen},
    journal ={J. Differential Geom.},
    year = {2008},
    volume = {78},
    number = {1},
    pages = {163--191},
    doi = {10.4310/jdg/1197320605}
}

@article{Oguiso09,
    author = {Oguiso, Keiji},
    title = {A remark on dynamical degrees of automorphisms of hyperkähler manifolds},
    journal = {Manuscripta Math.},
    year = {2009},
    volume={130},
    pages={101--111}
}

@misc{ASV26,
 author = {Amerik, Ekaterina and Soldatenkov, Andrey and Verbitsky, Misha},
 title = {Reflective lattices and hyperk\"{a}hler manifolds},
 year = {2026},
 howpublished = {Preprint, {arXiv}:2605.28593 [math.{AG}]},
 url = {https://arxiv.org/abs/2605.28593},
 arXiv = {arXiv:2605.28593}
}

@article{Huybrechts:Basic,
    author ={Huybrechts, Daniel},
    title ={Compact hyperkähler manifolds: basic results},
    journal ={Invent. Math.},
    year = {1999},
    volume={135},
    pages={63–-113}
}

@misc{DORV26,
      title={Birational automorphism groups in families of hyper-{K}\"ahler manifolds}, 
      author={Francesco Antonio Denisi and Claudio Onorati and Francesca Rizzo and Sasha Viktorova},
      year={2026},
      howpublished = {Preprint, {arXiv}:2601.16090 [math.{AG}]},
      url={https://arxiv.org/abs/2601.16090},
      arXiv = {arXiv:2601.16090}
}

@book{Serre:Arithmetique,
    author = {Serre, Jean--Pierre},
    title = {A course in arithmetic},
    publisher = {Graduate Texts in Mathematics, Springer-Verlag},
    volume = {7},
    year = {1973}
}

@article{XunYu,
    author = {Yu, Xun},
    title = {{K3} surface entropy and automorphism groups},
    journal = {J. Algebraic Geom.},
    year = {2025},
    volume = {34},
    pages = {205--231}
}
\bibliographystyle{alpha}
\end{document}